\documentclass[preprint,11pt,a4paper,oneside,reqno]{elsarticle}%
\usepackage{amsmath}
\usepackage{amsfonts}
\usepackage{amssymb}
\usepackage{graphicx}
\usepackage[margin=0.9in]{geometry}%
\providecommand{\U}[1]{\protect\rule{.1in}{.1in}}
\newtheorem{theorem}{Theorem}

\newtheorem{definition}{Definition}
\newtheorem{example}{Example}

\newtheorem{lemma}{Lemma}

\newtheorem{remark}{Remark}

\newproof{proof}{Proof}
\numberwithin{equation}{section}
\begin{document}
\begin{frontmatter}
\title{Almost automorphic solutions of discrete delayed neutral system\tnoteref{t1}}
\tnotetext[t1]{This study is supported by The Scientific and Technological Research Council
of Turkey (Grant number 1649B031101152).}
\author{Murat Ad\i var\corref{cor1}}
\ead{murat.adivar@ieu.edu.tr}
\address{ Izmir University of Economics\\
Department of Mathematics, 35330, Izmir Turkey}
\cortext[cor1]{Corresponding author}
\author{Halis Can Koyuncuo\u{g}lu}
\ead{can.koyuncuoglu@ieu.edu.tr}
\address{ Izmir University of Economics\\
Department of Mathematics, 35330, Izmir Turkey}

\begin{abstract}
We study almost automorphic solutions of the discrete delayed neutral dynamic
system%
\[
x(t+1)=A(t)x(t)+\Delta Q(t,x(t-g(t)))+G(t,x(t),x(t-g(t)))
\]
by means of a fixed point theorem due to Krasnoselskii. Using discrete variant of exponential
dichotomy and proving uniqueness of projector of discrete exponential dichotomy we invert the equation and
obtain some limit results leading to sufficient conditions for the existence
of almost automorphic solutions of the neutral system. Unlike the existing
literature we prove our existence results without assuming boundedness of
inverse matrix $A\left(  t\right)  ^{-1}$. Hence, we significantly improve the
results in the existing literature. We provide two examples to illustrate
effectiveness of our results. Finally, we also provide an existence result for
almost periodic solutions of the system.

\end{abstract}
\begin{keyword}

Almost automorphic\sep almost periodic\sep discrete exponential
dichotomy\sep discrete nonlinear neutral system\sep Krasnoselskii\sep unique projection

\MSC[2010] Primary 39A10\sep 39A12  Secondary 39A13 \sep 39B72
\end{keyword}
\end{frontmatter}

\section{Introduction}

The theory of neutral type equations has a significant application potential
in certain fields of applied mathematics, biology, and physics dealing with
modelling and controlling the dynamics of real life processes (see
\cite{beddington}, \cite{DeAngelis}, \cite{Fan}, \cite{Wang}, and references
therein). In particular, investigation of periodic solutions of neutral
dynamic systems has a special importance for researchers interested in
biological models of certain type of populations having periodical structures
(see \cite{Braverman}, \cite{Liu}, and \cite{Lu}). There is a vast literature
on stability analysis, oscillation theory, and periodic solutions of neutral
differential and neutral functional equations (see e.g. \cite{Adivar2},
\cite{Islam}, \cite{Raffoul1}, \cite{Raffoul2}, \cite{Sficas}, and \cite{Yu}).
We may refer to \cite{Adivar}, \cite{Kaufmann}, \cite{Raffoul},
\cite{Raffoul3}, and \cite{Raffoul4} for studies handling neutral difference
and neutral dynamic equations on time scales.

Periodicity may be a strong restriction in some specific real life models
including functions not strictly periodic but having values close enough to
each other at every different period. Many mathematical models (see e.g.
\cite{Hioe}, \cite{ohta1}, and \cite{ohta2}) in signal processing and
astrophysics require the use of almost periodic functions, informally, being a
nearly periodic functions where any one period is virtually identical to its
adjacent periods but not necessarily similar to periods much farther away in
time. The theory of almost periodic functions was first introduced by H. Bohr
\cite{bohr} and generalized by A. S. Besicovitch, W. Stepanoff, S. Bochner, and
J. von Neumann at the beginning of 20th century (see \cite{besitovitch
0}, \cite{Bochner}, \cite{Bochner}, \cite{bochner and neumann}, and
\cite{Stepanov}). The idea of almost periodicity can roughly be regarded as a
relaxation of strict the periodicity notion. A continuous function
$f:\mathbb{R\rightarrow R}$ is said to be almost periodic if the following
characteristic property holds:

$\boldsymbol{A.}$ \textit{For any }$\varepsilon>0$,\textit{ the set}%
\[
E\left(  \varepsilon,f(x)\right)  :=\left\{  \tau:\left\vert f\left(
x+\tau\right)  -f(x)\right\vert <\varepsilon\text{ for all }x\in
\mathbb{R}\right\}
\]
\textit{is relatively dense in the real line }$\mathbb{R}$\textit{. That is,
for any }$\varepsilon>0$, there exists a number $l\left(  \varepsilon\right)
>0$ such that any interval of length $l\left(  \varepsilon\right)  $ contains
a number in $E\left(  \varepsilon,f(x)\right)  $.

Afterwards, S. Bochner showed that almost periodicity is equivalent to the
following characteristic property which is also called \textit{the normality
condition}:

$\boldsymbol{B.}$ \textit{From any sequence of the form }$\left\{  f\left(
x+h_{n}\right)  \right\}  ,$\textit{ where }$h_{n}$\textit{ are real numbers,
one can extract a subsequence converging uniformly on the real line (see
\cite{besicovitch}, \cite{bochner 0} and \cite{corduneanu}).}

Obviously, every periodic function is almost periodic. However, there exist
almost periodic functions which are not periodic. For instance, the function%
\[
f\left(  t\right)  =e^{it}+e^{i\pi t}%
\]
is almost periodic but there is no any real number $\omega\neq0$ such that
$f\left(  t+\omega\right)  =f\left(  t\right)  $, since the functions $e^{it}$
and $e^{i\pi t}$ are linearly independent.

Theory of almost automorphic functions was first studied by S. Bochner
\cite{Bochner}. It is a property of a function which can be obtained by
replacing convergence with uniform convergence in normality condition (B).
More explicitly, a continuous function $f:\mathbb{R\rightarrow R}$ is said to
be almost automorphic if for every sequence $\left\{  h_{n}^{\prime}\right\}
_{n\in\mathbb{Z}_{+}}$ of real numbers there exists a subsequence $\left\{
h_{n}\right\}  $ such that $\lim_{m\rightarrow\infty}\lim_{n\rightarrow\infty
}f(t+h_{n}-h_{m})=f(t)$ for each $t\in\mathbb{R}$. For more reading on almost
automorphic functions, we refer to \cite{Diagana} and \cite{Nguerekata}.
Obviously, almost periodicity implies almost automorphicity but not vice versa. It
is shown in \cite{veech2} that the function%
\[
f\left(  t\right)  =\frac{2+\exp\left(  it\right)  +\exp\left(  i\sqrt
{2}t\right)  }{\left\vert 2+\exp\left(  it\right)  +\exp\left(  i\sqrt
{2}t\right)  \right\vert }\text{, }t\in\mathbb{R}%
\]
is almost automorphic but not almost periodic.

Unlike the vast literature on almost periodicity, there is a poor research
backlog on almost automorphic solutions of difference equations. To the best
of our knowledge the study of almost automorphic solutions of difference
equations was begun by Araya et al. in \cite{Araya}. Afterwards, C. Lizama
and J. G. Mesquita \cite{Lizama2} studied almost automorphic solutions of
non-autonomous difference equations%
\[
u\left(  k+1\right)  =A\left(  k\right)  u\left(  k\right)  +f\left(
k,u\left(  k\right)  \right)  ,\ \ k\in\mathbb{Z}.
\]
Employing exponential dichotomy and contraction mapping principle they
proposed some existence results. Furthermore, in \cite{Lizama1} C. Lizama and
J. G. Mesquita perfectly generalized the notion of almost automorphy by
studying of almost automorphic solutions of dynamic equations on time scales
that are invariant under translation. In \cite{Castillo}, S. Castillo and M.
Pinto studied almost automorphic solutions of the system with constant
coefficient matrix $A$%
\[
y\left(  n+1\right)  =Ay\left(  n\right)  +f\left(  n\right)
\]
using $\left(  \mu_{1},\mu_{2}\right)  $-exponential dichotomy. In
\cite{Mishra} I. Mishra et al. investigated almost automorphic solutions to
functional differential equation%
\begin{equation}
\frac{d}{dt}\left(  x\left(  t\right)  -F_{1}\left(  t,x\left(  t-g\left(
t\right)  \right)  \right)  \right)  =A\left(  t\right)  x\left(  t\right)
+F_{2}\left(  t,x\left(  t\right)  ,x\left(  t-g\left(  t\right)  \right)
\right)  \label{cont_eq}%
\end{equation}
using the theory of evolution semigroup. Note that almost periodic solutions
of Eq. (\ref{cont_eq}) have also been studied in \cite{Abbas1} by means of the
theory of evolution semigroup.

In the present work, we propose some existence results for almost automorphic
solutions of the discrete neutral delayed system%
\begin{equation}
x(t+1)=A(t)x(t)+\Delta Q(t,x(t-g(t)))+G(t,x(t),x(t-g(t))) \label{system 0}%
\end{equation}
by using fixed point theory. The highlights of the paper can be summarized as follows:

\begin{itemize}
\item In our analysis, we prefer using exponential dichotomy instead of theory
of evolution semigroup since the conditions required by theory of evolution
are strict and not easy to check (for regarding discussion see \cite{Chen}). We prove uniqueness of projector of discrete exponential dichotomy. This result has a wide application potential in theory of difference equations.

\item In \cite[Relations (3.10) and (3.11)]{Lizama2}, the authors obtain the
limiting properties of exponential dichotomy by using the product integral on
discrete domain (see \cite[Section 4]{Slavik} and \cite[Section 4]{Lizama1}).
This method requires boundedness of inverse matrix $A(t)^{-1}$ as a compulsory
condition. Different than \cite{Lizama2}, we obtain our limit results
without assuming boundedness of the inverse matrix $A(t)^{-1}$ (see Theorem
\ref{thm projection unique}).

\item Using a different approach we improve the existence results
\cite[Theorem 3.1 and Theorem 4.3]{Lizama2} (see Example \ref{ex comparison})
and extend the results of \cite{Castillo} to the systems with nonconstant
coefficient matrix $A\left(  t\right)  $. Two examples are given to illustrate
the effectiveness of our results.
\end{itemize}

The latter part of the paper is organized as follows: In the next section, we
give basic definitions and properties of discrete almost automorphic functions
and prove our limit results regarding discrete exponential dichotomy. In the
final section, we propose some sufficient conditions for the existence of
almost automorphic and almost periodic solutions of the system (\ref{system 0}%
) by means of Krasnoselskii's fixed point theorem.

\section{Almost automorphic functions and exponential dichotomy}

Let $\mathcal{X}$ be a (real or complex) Banach space endowed with the norm
$\left\Vert .\right\Vert _{\mathcal{X}}$ and $\mathcal{B}(\mathcal{X})$ is a
Banach space of all bounded linear operators from $\mathcal{X}$ to
$\mathcal{X}$ with the norm $\left\Vert .\right\Vert _{\mathcal{B}%
(\mathcal{X})}$ given by%
\[
\left\Vert L\right\Vert _{\mathcal{B}(\mathcal{X})}:=\sup\left\{  \left\Vert
Lx\right\Vert _{\mathcal{X}}:x\in\mathcal{X}\text{ and }\left\Vert
x\right\Vert _{\mathcal{X}}\leq1\right\}  .
\]

Following definitions and results can be found in \cite{Araya} and
\cite{Lizama2}.

\begin{definition}
A function $f:\mathbb{Z\rightarrow}\mathcal{X}$ is said to be discrete almost
automorphic if for every integer sequence $\left\{  k_{n}^{\prime}\right\}
_{n\in\mathbb{Z}_{+}}$ there exists a subsequence $\left\{  k_{n}\right\}
_{n\in\mathbb{Z}_{+}}$ such that%
\begin{equation}
\lim_{n\rightarrow\infty}f(t+k_{n})=:\bar{f}(t) \label{fbar1}%
\end{equation}
is well defined for each $t\in\mathbb{Z}$ and%
\begin{equation}
\lim_{n\rightarrow\infty}\bar{f}(t-k_{n})=f(t) \label{fbar2}%
\end{equation}
for each $t\in\mathbb{Z}$.
\end{definition}

\begin{definition}
A function $g:\mathbb{Z\times}\mathcal{X}\mathbb{\rightarrow}\mathcal{X}$ of
two variables is said to be discrete almost automorphic in $t\in\mathbb{Z}$
for each $x\in\mathcal{X},$ if for every integer sequence $\left\{
k_{n}^{\prime}\right\}  _{n\in\mathbb{Z}_{+}},$ there exists a subsequence
$\left\{  k_{n}\right\}  _{n\in\mathbb{Z}_{+}}$ such that%
\[
\lim_{n\rightarrow\infty}g(t+k_{n},x)=:\bar{g}(t,x)
\]
is well defined for each $t\in\mathbb{Z}$, $x\in\mathcal{X}$ and
\[
\lim_{n\rightarrow\infty}\bar{g}(t-k_{n},x)=:g(t,x)
\]
for each $t\in\mathbb{Z}$ and $x\in\mathcal{X}.$
\end{definition}

Throughout the paper, $\mathcal{A}(\mathcal{X})$ represents the set of
discrete almost automorphic functions taking values on $\mathcal{X}$. Notice
that $\mathcal{A}\left(  \mathcal{X}\right)  $ is a Banach space endowed by
the norm%
\[
\left\Vert f\right\Vert _{\mathcal{A}(\mathcal{X})}:=\sup_{t\in\mathbb{Z}%
}\left\Vert f(t)\right\Vert _{\mathcal{X}}.
\]
Some properties of discrete almost automorphic functions are listed in the
following theorems:

\begin{theorem}
\cite{Araya}\label{thm properties of almost automorphic} Let $f_{1}%
,f_{2}:\mathbb{Z\rightarrow}\mathcal{X}$ and $g_{1},g_{2}:\mathbb{Z\times
}\mathcal{X}\mathbb{\rightarrow}\mathcal{X}$ be discrete almost automorphic
functions in $t\in\mathbb{Z},$ then

\begin{description}
\item[i.] $f_{1}+f_{2}$ is discrete almost automorphic

\item[ii.] $g_{1}+g_{2}$ is discrete almost automorphic in $t$ for each
$x\in\mathcal{X}$

\item[iii.] $cf_{1}$ is discrete almost automorphic for every scalar $c$

\item[iv.] For every scalar $c,$ $cg_{1}$ is discrete almost automorphic in $t$
for each $x\in\mathcal{X}$

\item[v.] For each fixed $k\in\mathbb{Z}$, the function $f_{1}\left(
t+k\right)  $ is discrete almost automorphic

\item[vi.] The function $\hat{f}_{1}:\mathbb{Z\rightarrow}\mathcal{X}$ defined by $\hat{f}_{1}(t):=f_{1}(-t)$ is discrete almost automorphic

\item[vii.] $\sup_{t\in\mathbb{Z}}\left\Vert f_{1}(t)\right\Vert _{\mathcal{X}%
}<\infty$ for each $t\in$ $\mathbb{Z}$

\item[viii.] $\sup_{t\in\mathbb{Z}}\left\Vert g_{1}(t,x)\right\Vert
_{\mathcal{X}}<\infty$ for each $t\in$ $\mathbb{Z}$ and $x\in\mathcal{X}$

\item[ix.] $\sup_{t\in\mathbb{Z}}\left\Vert \bar{f}_{1}(t)\right\Vert
_{\mathcal{X}}\leq\sup_{t\in\mathbb{Z}}\left\Vert f_{1}(t)\right\Vert
_{\mathcal{X}}$ for all $t\in$ $\mathbb{Z}$ where $\bar{f}_{1}$ is defined as in
(\ref{fbar1})

\item[x.] $\sup_{t\in\mathbb{Z}}\left\Vert \bar{g}_{1}(t,x)\right\Vert
_{\mathcal{X}}<\infty$ for each $t\in$ $\mathbb{Z}$ and $x\in\mathcal{X}.$
\end{description}
\end{theorem}

\begin{theorem}
\label{thmlipschitz} \cite{Araya} Let $g:\mathbb{Z\times}\mathcal{X}%
\mathbb{\rightarrow}\mathcal{X}$ be discrete almost automorphic in
$t\in\mathbb{Z},$ for each $x\in\mathcal{X}$ satisfying Lipschitz condition in
$x$ uniformly in $t,$ that is%
\[
\left\Vert g(t,x)-g(t,y)\right\Vert _{\mathcal{A}(\mathcal{X})}\leq
L\left\Vert x-y\right\Vert _{\mathcal{X}},\text{ }\forall x,y\in\mathcal{X}.
\]
Suppose $\varphi:\mathbb{Z\rightarrow}\mathcal{X}$ is discrete almost
automorphic function, then the function $g\left(  t,\varphi\left(  t\right)
\right)  $ is discrete almost automorphic.
\end{theorem}

\begin{definition}
[Discrete exponential dichotomy]\label{def exp dichotomy} Let $X(t)$ be the
principal fundamental matrix solution of the linear homogeneous system%
\begin{equation}
x(t+1)=A(t)x(t),\text{ }x(t_{0})=x_{0}. \label{1}%
\end{equation}
Then (\ref{1}) is said to admit an exponential dichotomy if there exist a
projection $P$ and positive constants $\alpha_{1},\alpha_{2},\beta_{1}$ and
$\beta_{2}$ such that%
\begin{align}
\left\Vert X(t)PX^{-1}(s)\right\Vert _{\mathcal{B}(\mathcal{X})}  &  \leq
\beta_{1}\left(  1+\alpha_{1}\right)  ^{s-t},\text{ }t\geq s,\label{2}\\
\left\Vert X(t)\left(  I-P\right)  X^{-1}(s)\right\Vert _{\mathcal{B}%
(\mathcal{X})}  &  \leq\beta_{2}\left(  1+\alpha_{2}\right)  ^{t-s},\text{
}s\geq t. \label{3}%
\end{align}

\end{definition}

\begin{remark}
Notice that in \cite{Bayliss} and \cite{Lizama2}, the discrete exponential
dichotomy is defined by using the exponential function $\exp\left(
\alpha\left(  s-t\right)  \right)  $ instead of the discrete exponential
function $e_{\alpha}(t,s)=\left(  1+\alpha\right)  ^{s-t}$ (satisfying
$\Delta_{t}e_{\alpha}(t,s)=\alpha e_{\alpha}(t,s)$) in (\ref{2}) and
(\ref{3}), respectively. For convenience we prefer using Definition
\ref{def exp dichotomy} which is evidently equivalent to \cite[Definition
2.11]{Lizama2}. Notice that Definition \ref{def exp dichotomy} is also
consistent with the unified version of exponential dichotomy (see
\cite[Definition 2.12]{Lizama1}) which covers both the discrete and continuous
cases.\ For more reading on exponential dichotomy we may refer to
\cite{Coppel}
\end{remark}

We prove the following lemmas for further use in our analysis:

\begin{lemma}
\label{Lem1}Let $\varphi:\mathbb{Z\rightarrow(}0,\infty\mathbb{)}$ and
$\psi:\mathbb{Z\rightarrow(}0,\infty\mathbb{)}$ be two functions satisfying
\begin{equation}
\varphi(t)%
{\displaystyle\sum\limits_{j=-\infty}^{t-1}}
\varphi(j)^{-1}\leq\mu,\text{ }t\in\mathbb{Z}, \label{4}%
\end{equation}%
\begin{equation}
\psi(t)%
{\displaystyle\sum\limits_{j=t}^{\infty}}
\psi(j)^{-1}\leq\gamma,\text{ }t\in\mathbb{Z}, \label{5}%
\end{equation}
for some constants $\mu>0$ and $\gamma>0$. Then for any $t_{0}\in\mathbb{Z}$,
there exist positive constants $c$ and $\widetilde{c}$ such that%
\[
\varphi(t)\leq c\left(  1+\mu^{-1}\right)  ^{t_{0}-t}\text{ for }t\geq t_{0}%
\]
and%
\[
\psi(t)\leq\widetilde{c}\left(  1+\gamma^{-1}\right)  ^{t-t_{0}}\text{ for
}t\leq t_{0}.
\]

\end{lemma}

\begin{proof}
Define the function
\[
u(t):=%
{\displaystyle\sum\limits_{j=-\infty}^{t-1}}
\left(  \varphi(j)\right)  ^{-1}%
\]
with $\Delta u(t)=\left(  \varphi(t)\right)  ^{-1}$, where $\Delta$ is the
forward difference operator. By (\ref{4}), we have%
\[
u(t)\leq\mu\Delta u(t),
\]
and hence,
\[
u(t)\geq\left(  1+\mu^{-1}\right)  ^{t-t_{0}}u\left(  t_{0}\right)
\text{,\ for\ }t\geq t_{0}.
\]
This implies%
\begin{align*}
\varphi(t)  &  \leq\mu u(t)^{-1}\\
&  \leq\mu u\left(  t_{0}\right)  ^{-1}\left(  1+\mu^{-1}\right)  ^{t_{0}-t}\\
&  \leq c\left(  1+\mu^{-1}\right)  ^{t_{0}-t}%
\end{align*}
for $c=\mu u\left(  t_{0}\right)  ^{-1}$ and $t\geq t_{0}$. Similarly, the
function
\[
v(t):=%
{\displaystyle\sum\limits_{j=t}^{\infty}}
\psi(j)^{-1},
\]
with $\Delta v(t)=-\psi(t)^{-1}$, satisfies
\[
\gamma\Delta v(t)\leq-v(t+1).
\]
Solving the last inequality for $t_{0}\geq t,$ we get%
\[
v\left(  t_{0}\right)  -v(t)\left(  1+\gamma^{-1}\right)  ^{t-t_{0}}%
\leq0\text{.}%
\]
By using (\ref{5}), we obtain%
\begin{align*}
\psi(t)  &  \leq\gamma v(t)^{-1}\\
&  \leq\gamma v\left(  t_{0}\right)  ^{-1}\left(  1+\gamma^{-1}\right)
^{t-t_{0}}\\
&  \leq\widetilde{c}\left(  1+\gamma^{-1}\right)  ^{t-t_{0}}%
\end{align*}
for $\widetilde{c}=\gamma v\left(  t_{0}\right)  ^{-1}$ and $t_{0}\geq t$. The
proof is complete.
\end{proof}

\begin{lemma}
\label{Lem2}If the system (\ref{1}) admits an exponential dichotomy, then
$x=0$ is the unique bounded solution of the system (\ref{1}).
\end{lemma}

\begin{proof}
Let $B_{0}$ be set of initial conditions $\xi$ belonging to bounded solutions
of (\ref{1}). Assume $\left(  I-P\right)  \xi\neq0$ and define $\phi
(t)^{-1}:=\left\Vert X(t)\left(  I-P\right)  \xi\right\Vert _{\mathcal{B}%
(\mathcal{X})}$. Using the equality $\left(  I-P\right)  ^{2}=I-P$ we get%
\[%
{\displaystyle\sum\limits_{j=t}^{\infty}}
X(t)\left(  I-P\right)  \xi\phi(j)=%
{\displaystyle\sum\limits_{j=t}^{\infty}}
X(t)\left(  I-P\right)  X^{-1}\left(  j\right)  X(j)\left(  I-P\right)
\xi\phi(j).
\]
Taking the norm of both sides, we obtain%
\begin{align*}
\phi(t)^{-1}%
{\displaystyle\sum\limits_{j=t}^{\infty}}
\phi(j)  &  \leq%
{\displaystyle\sum\limits_{j=t}^{\infty}}
\left\Vert X(t)\left(  I-P\right)  X^{-1}\left(  j\right)  \right\Vert
_{\mathcal{B}(\mathcal{X})}\phi^{-1}(j)\phi(j)\\
&  \leq%
{\displaystyle\sum\limits_{j=t}^{\infty}}
\beta_{2}\left(  1+\alpha_{2}\right)  ^{t-j}\\
&  =%
{\displaystyle\sum\limits_{j=0}^{\infty}}
\beta_{2}\left(  1+\alpha_{2}\right)  ^{-j}\\
&  =\beta_{2}\frac{1+\alpha_{2}}{\alpha_{2}}.
\end{align*}
This yields
\[%
{\displaystyle\sum\limits_{j=t}^{\infty}}
\phi(j)\leq\phi(t)\beta_{2}\frac{1+\alpha_{2}}{\alpha_{2}},
\]
uniformly in $t$. Hence, we get
\[
\lim\underset{j\in\left[  t,\infty\right)  }{\inf}\phi(j)=0.
\]
which means that $\left\Vert X(t)\left(  I-P\right)  \xi\right\Vert
_{\mathcal{B}(\mathcal{X})}$ have to be unbounded.

Similarly, if we assume that $P\xi\neq0$, define $\left(  \theta\left(
t\right)  \right)  ^{-1}=\left\Vert X(t)P\xi\right\Vert _{\mathcal{B}%
(\mathcal{X})}$, and repeat the above procedure, we get%
\[%
{\displaystyle\sum\limits_{j=-\infty}^{t-1}}
\theta(j)X(t)P\xi=%
{\displaystyle\sum\limits_{j=-\infty}^{t-1}}
X(t)PX^{-1}\left(  j\right)  X(j)P\xi\theta(j)
\]
since $P^{2}=P$. Taking the norm of both sides we obtain%
\begin{align*}
\left(  \theta\left(  t\right)  \right)  ^{-1}%
{\displaystyle\sum\limits_{j=-\infty}^{t-1}}
\theta(j)  &  =%
{\displaystyle\sum\limits_{j=-\infty}^{t-1}}
\left\Vert X(t)PX^{-1}\left(  j\right)  \right\Vert _{\mathcal{B}%
(\mathcal{X})}\theta^{-1}(j)\theta(j)\\
&  \leq%
{\displaystyle\sum\limits_{j=-\infty}^{t-1}}
\beta_{1}\left(  1+\alpha_{1}\right)  ^{j-t}\\
&  \leq\frac{\beta_{1}}{\alpha_{1}}%
\end{align*}
and%
\[%
{\displaystyle\sum\limits_{j=-\infty}^{t-1}}
\theta(j)\leq\theta\left(  t\right)  \frac{\beta_{1}}{\alpha_{1}}.
\]
This shows that
\[
\lim\underset{j\in\left[  -\infty,t-1\right)  }{\inf}\theta(j)=0
\]
and hence $\left\Vert X(t)P\xi\right\Vert _{\mathcal{B}(\mathcal{X})}$ must be
unbounded. Consequently, boundedness of a solution of the system (\ref{1}) is
possible only if $B_{0}=\left\{  0\right\}  $, which means, if $x\left(
t\right)  $ is a bounded solution of (\ref{1}), then $x(t)=0$. The proof is complete.
\end{proof}

\begin{theorem}
If the homogeneous system (\ref{1}) admits an exponential dichotomy, then the
projection $P$ of the exponential dichotomy is unique.
\end{theorem}

\begin{proof}
Suppose that the system (\ref{1}) admits an exponential dichotomy. At first,
we need to show that $\left\Vert X(t)P\right\Vert _{\mathcal{B}(\mathcal{X})}$
is bounded for $t\geq t_{0}$ and $\left\Vert X(t)\left(  I-P\right)
\right\Vert _{\mathcal{B}(\mathcal{X})}$ is bounded for $t\leq t_{0}.$ Define
the function $\varphi(t):=\left\Vert X(t)P\right\Vert _{\mathcal{B}%
(\mathcal{X})}$ and consider the following equality%
\[%
{\displaystyle\sum\limits_{j=-\infty}^{t-1}}
X(t)P\varphi(j)^{-1}=%
{\displaystyle\sum\limits_{j=-\infty}^{t-1}}
X(t)PX^{-1}(j)X(j)P\varphi(j)^{-1}.
\]
Taking the norm of both sides, we get%
\[
\varphi(t)%
{\displaystyle\sum\limits_{j=-\infty}^{t-1}}
\varphi(j)^{-1}\leq\frac{\beta_{1}}{\alpha_{1}}:=K.
\]
Employing Lemma \ref{Lem1}, there exists a positive constant $c$ such that%
\[
\varphi(t)\leq c\left(  1+K^{-1}\right)  ^{t_{0}-t}\text{ for }t\geq t_{0},
\]
which means $\left\Vert X(t)P\right\Vert _{\mathcal{B}(\mathcal{X})}$ is
bounded for $t\geq t_{0}$.

Performing the similar procedure for the function $\psi(t):=\left\Vert
X(t)\left(  I-P\right)  \right\Vert _{\mathcal{B}(\mathcal{X})}$ we get%
\[%
{\displaystyle\sum\limits_{j=t}^{\infty}}
X(t)\left(  I-P\right)  \left(  \psi(j)\right)  ^{-1}=%
{\displaystyle\sum\limits_{j=t}^{\infty}}
X(t)\left(  I-P\right)  X^{-1}(j)X(j)\left(  I-P\right)  \left(
\psi(j)\right)  ^{-1},
\]
which implies that%
\[
\psi(t)%
{\displaystyle\sum\limits_{j=t}^{\infty}}
\psi(j)^{-1}\leq\beta_{2}\frac{1+\alpha_{2}}{\alpha_{2}}:=\hat{K}.
\]
By Lemma \ref{Lem1}, we can find a constant $\widehat{c}>0$ such that%
\[
\psi(t)\leq\widehat{c}\left(  1+\hat{K}^{-1}\right)  ^{t-t_{0}}\text{ for
}t_{0}\geq t,
\]
This shows that $\left\Vert X(t)\left(  I-P\right)  \right\Vert _{\mathcal{B}%
(\mathcal{X})}$ is bounded for $t_{0}\geq t$.

Suppose that there exists another projection $\tilde{P}$ $\neq P$ of
exponential dichotomy of (\ref{1}). Using the similar arguments we may find
constants $N$ and $\widetilde{N}$ such that%
\[
\left\Vert X(t)\tilde{P}\right\Vert _{\mathcal{B}(\mathcal{X})}\leq N\text{
for }t\geq t_{0},
\]
and%
\[
\left\Vert X(t)\left(  I-\tilde{P}\right)  \right\Vert _{\mathcal{B}%
(\mathcal{X})}\leq\ \widetilde{N}\text{ for }t_{0}\geq t.
\]
Using (\ref{2}-\ref{3}), for any arbitrary nonzero vector $\xi,$ we get%
\begin{align*}
\left\Vert X(t)P\left(  I-\tilde{P}\right)  \xi\right\Vert _{\mathcal{B}%
(\mathcal{X})}  &  =\left\Vert X(t)PX^{-1}(t_{0})X(t_{0})\left(  I-\tilde
{P}\right)  \xi\right\Vert _{\mathcal{B}(\mathcal{X})}\\
&  \leq\left\Vert X(t)PX^{-1}(t_{0})\right\Vert _{\mathcal{B}(\mathcal{X}%
)}\left\Vert X(t_{0})\left(  I-\tilde{P}\right)  \xi\right\Vert _{\mathcal{B}%
(\mathcal{X})}\\
&  \leq\beta_{1}\left\Vert \left(  I-\tilde{P}\right)  \xi\right\Vert
_{\mathcal{X}}\text{ for }t\geq t_{0}%
\end{align*}
and%
\begin{align*}
\left\Vert X(t)P\left(  I-\tilde{P}\right)  \xi\right\Vert _{\mathcal{B}%
(\mathcal{X})}  &  =\left\Vert X(t)PX^{-1}(t)X(t)\left(  I-\tilde{P}\right)
X^{-1}(t_{0})X(t_{0})\left(  I-\tilde{P}\right)  \xi\right\Vert _{\mathcal{B}%
(\mathcal{X})}\\
&  \leq\left\Vert X(t)PX^{-1}(t)\right\Vert _{\mathcal{B}(\mathcal{X}%
)}\left\Vert X(t)\left(  I-\tilde{P}\right)  X^{-1}(t_{0})\right\Vert
_{\mathcal{B}(\mathcal{X})}\left\Vert X(t_{0})\left(  I-\tilde{P}\right)
\xi\right\Vert _{\mathcal{B}(\mathcal{X})}\\
&  \leq\beta_{1}\beta_{2}\left\Vert \left(  I-\tilde{P}\right)  \xi\right\Vert
_{\mathcal{X}}\text{ for }t_{0}\geq t
\end{align*}
since $X(t_{0})=I$. Then $x(t)=X(t)P(I-\tilde{P})\xi$ is bounded solution of
(\ref{1}). Observe that $x(t)=X(t)\left(  I-P\right)  \tilde{P}\xi$ is also a
bounded solution of (\ref{1}). Employing Lemma \ref{Lem2}, we get $x=0,$ and
hence, $P=P\tilde{P}=\tilde{P}$. The proof is complete.
\end{proof}

\begin{theorem}
\label{thm projection unique}Suppose that the system (\ref{1}) admits an
exponential dichotomy with the projection $P$ and the positive constants
$\alpha_{1},\alpha_{2}$, $\beta_{1}$, and $\beta_{2}$. Let the matrix valued
function $A(t)$ in (\ref{1}) be almost automorphic. That is, for any sequence
$\{\widetilde{\theta}_{k}\}_{k\in\mathbb{Z}_{+}}$ of integers there exists a
subsequence $\{\theta_{k}\}_{k\in\mathbb{Z}_{+}}$ such that%
\[
\lim_{k\rightarrow\infty}A(t+\theta_{k}):=\bar{A}(t)
\]
is well defined and%
\[
\lim_{k\rightarrow\infty}\bar{A}(t-\theta_{k})=A(t)
\]
for each $t\in\mathbb{Z}$. Then
\begin{equation}
\lim_{k\rightarrow\infty}X(t+\theta_{k})PX^{-1}(s+\theta_{k}):=\bar{X}%
(t)\bar{P}\bar{X}^{-1}(s)\text{ for }s\in(-\infty,t]\cap\mathbb{Z}
\label{forward limit1}%
\end{equation}
and%
\begin{equation}
\lim_{k\rightarrow\infty}X(t+\theta_{k})\left(  I-P\right)  X^{-1}%
(s+\theta_{k}):=\bar{X}(t)\left(  I-\bar{P}\right)  \bar{X}^{-1}(s)\text{ for
}s\in\lbrack t,\infty)\cap\mathbb{Z} \label{forward limit 2}%
\end{equation}
are well defined for each $t\in\mathbb{Z}$ and the limiting system%
\begin{equation}
x(t+1)=\bar{A}(t)x(t),\text{ }x(t_{0})=x_{0} \label{lim sys}%
\end{equation}
admits an exponential dichotomy with the projection $\bar{P}$ and the same
constants. Furthermore, for each $t\in\mathbb{Z}$ we have%
\begin{equation}
\lim_{k\rightarrow\infty}\bar{X}(t-\theta_{k})\bar{P}\bar{X}^{-1}(s-\theta
_{k})=X(t)PX^{-1}(s)\text{, \ }s\in(-\infty,t]\cap\mathbb{Z} \label{back1}%
\end{equation}
and%
\begin{equation}
\lim_{k\rightarrow\infty}\bar{X}(t-\theta_{k})\left(  I-\bar{P}\right)
\bar{X}^{-1}(s-\theta_{k})=X(t)\left(  I-P\right)  X^{-1}(s)\text{, }%
s\in\lbrack t,\infty)\cap\mathbb{Z}. \label{back2}%
\end{equation}

\end{theorem}

\begin{proof}
We first show that the sequence $\left\{  X(t_{0}+\theta_{k})PX^{-1}%
(t_{0}+\theta_{k})\right\}  $ is convergent. Suppose the contrary, then there
exist two subsequences
\[
\left\{  X(t_{0}+\theta_{k_{m}})PX^{-1}(t_{0}+\theta_{k_{m}})\right\}
\ \ \text{and\ \ }\left\{  X(t_{0}+\theta_{k_{m}^{\prime}})PX^{-1}%
(t_{0}+\theta_{k_{m}^{\prime}})\right\}
\]
converging two different
numbers $\overline{P}$ and \underline{$P$}$,$ respectively. From (\ref{2}) we
have%
\begin{equation}
\left\Vert X(t+\theta_{k_{m}})PX^{-1}(s+\theta_{k_{m}}))\right\Vert
_{\mathcal{B}(\mathcal{X})}\leq\beta_{1}\left(  1+\alpha_{1}\right)
^{s-t},\text{ }t\geq s. \label{4.0}%
\end{equation}
and
\begin{equation}
\left\Vert X(t+\theta_{k_{m}^{\prime}})PX^{-1}(s+\theta_{k_{m}^{\prime}%
}))\right\Vert _{\mathcal{B}(\mathcal{X})}\leq\beta_{1}\left(  1+\alpha
_{1}\right)  ^{s-t},\text{ }t\geq s. \label{5.0}%
\end{equation}
Let $X_{k_{m}}(t)$ and $X_{k_{m}^{\prime}}(t)$ denote the principal
fundamental matrix solutions of the systems:%
\begin{equation}
x(t+1)=A(t+\theta_{k_{m}})x(t),\text{ }x(t_{0})=x_{0}, \label{6}%
\end{equation}
and%
\begin{equation}
x(t+1)=A(t+\theta_{k_{m}^{\prime}})x(t),\text{ }x(t_{0})=x_{0}, \label{7}%
\end{equation}
respectively. Then we must have
\begin{equation}
X(t+\theta_{k_{m}})=X_{k_{m}}(t)X(t_{0}+\theta_{k_{m}}). \label{6.0}%
\end{equation}
since%
\[
\Delta\left[  X_{k_{m}}(t)^{-1}X(t+\theta_{k_{m}})\right]  =0.
\]
Similarly, we get
\begin{equation}
X(t+\theta_{k_{m}^{\prime}})=X_{k_{m}^{\prime}}(t)X(t_{0}+\theta
_{k_{m}^{\prime}}). \label{7.0}%
\end{equation}
Since $A(t+\theta_{k_{m}})\rightarrow\bar{A}(t)$, $A(t+\theta_{k_{m}^{\prime}%
})x(t)\rightarrow\bar{A}(t)x(t)$ as $m\rightarrow\infty$ for each
$t\in\mathbb{Z}$, we have%
\[
A(t+\theta_{k_{m}})x(t)\rightarrow\bar{A}(t)x(t),
\]
and%
\[
A(t+\theta_{k_{m}^{\prime}})x(t)\rightarrow\bar{A}(t)x(t).
\]
Thus, the sequences $\left\{  X_{k_{m}}(t)\right\}  $ and $\left\{
X_{k_{m}^{\prime}}(t)\right\}  $ converge to $\bar{X}(t)$ as $m\rightarrow
\infty$ for each $t\in\mathbb{Z}$. Now, the exponential dichotomy of the
linear homogeneous system (\ref{1}) plays a crucial role. Using (\ref{6.0})
along with (\ref{4.0}) and (\ref{5.0}), we get%
\[
\left\Vert X_{k_{m}}(t)X(t_{0}+\theta_{k_{m}})PX^{-1}(t_{0}+\theta_{k_{m}%
})X^{-1}_{k_{m}}(s)\right\Vert _{\mathcal{B}(\mathcal{X})}\leq\beta
_{1}\left(  1+\alpha_{1}\right)  ^{s-t},\text{ }t\geq s
\]
and%
\[
\left\Vert X_{k_{m}^{\prime}}(t)X(t_{0}+\theta_{k_{m}^{\prime}})PX^{-1}%
(t_{0}+\theta_{k_{m}^{\prime}})X^{-1}_{k_{m}^{\prime}}(s)\right\Vert
_{\mathcal{B}(\mathcal{X})}\leq\beta_{1}\left(  1+\alpha_{1}\right)
^{s-t},\text{ }t\geq s.
\]
Taking the limit as $m\rightarrow\infty,$ we obtain%
\begin{align}
\left\Vert \bar{X}(t)\overline{P}\bar{X}^{-1}(s)\right\Vert _{\mathcal{B}%
(\mathcal{X})}  &  \leq\beta_{1}\left(  1+\alpha_{1}\right)  ^{s-t},\text{
}t\geq s,\label{12}\\
\left\Vert \bar{X}(t)\underline{P}\bar{X}^{-1}(s)\right\Vert _{\mathcal{B}%
(\mathcal{X})}  &  \leq\beta_{1}\left(  1+\alpha_{1}\right)  ^{s-t},\text{
}t\geq s. \label{13}%
\end{align}
Applying the similar procedure we arrive at the following inequalities%
\begin{align}
\left\Vert \bar{X}(t)\left(  I-\overline{P}\right)  \bar{X}^{-1}(s)\right\Vert
_{\mathcal{B}(\mathcal{X})}  &  \leq\beta_{2}\left(  1+\alpha_{2}\right)
^{t-s},\text{ }s\geq t,\label{14}\\
\left\Vert \bar{X}(t)\left(  I-\underline{P}\right)  \bar{X}^{-1}%
(s)\right\Vert _{\mathcal{B}(\mathcal{X})}  &  \leq\beta_{2}\left(
1+\alpha_{2}\right)  ^{t-s},\text{ }s\geq t. \label{15}%
\end{align}
Inequalities (\ref{12}-\ref{15}) show that the limiting system (\ref{lim sys})
admits exponential dichotomy and both $\overline{P}$ and \underline{$P$} are
projections. By Theorem \ref{thm projection unique} we conclude that
$\overline{P}=$\underline{$P$}. This means the sequence $\left\{  X(t_{0}+\theta
_{k})PX^{-1}(t_{0}+\theta_{k})\right\}  $ is convergent as desired. Assume
that $X(t_{0}+\theta_{k})PX^{-1}(t_{0}+\theta_{k})\rightarrow\overline{P}$ and that
$X_{k}(t)$ is the principal fundamental matrix solution of the system%
\[
x(t+1)=A(t+\theta_{k})x(t),\text{ }x(t_{0})=x_{0}.
\]
Then $X_{k}(t)\rightarrow\bar{X}(t)$ and $X_{k}^{-1}(s)\rightarrow\bar{X}%
^{-1}(s)$ as $k\rightarrow\infty$ for each $t,s\in\mathbb{Z}$. This means for
each $t\in\mathbb{Z}$%
\[
X(t+\theta_{k})PX^{-1}(s+\theta_{k})\rightarrow\bar{X}(t)\overline{P}\bar
{X}^{-1}(s)\text{ for }s\in(-\infty,t]\cap\mathbb{Z}%
\]
and%
\[
X(t+\theta_{k})\left(  I-P\right)  X^{-1}(s+\theta_{k})\rightarrow\bar
{X}(t)\left(  I-\overline{P}\right)  \bar{X}^{-1}(s)\text{ for }s\in\lbrack
t,\infty)\cap\mathbb{Z}.
\]
Hence, we prove (\ref{forward limit1}) and (\ref{forward limit 2}). From
(\ref{4.0}) and (\ref{5.0}) we also have%
\[
\left\Vert X(t+\theta_{k})PX^{-1}(s+\theta_{k}))\right\Vert _{\mathcal{B}%
(\mathcal{X})}\leq\beta_{1}\left(  1+\alpha_{1}\right)  ^{s-t},\text{ }t\geq
s
\]
and similarly,%
\[
\left\Vert X(t+\theta_{k})\left(  I-P\right)  X^{-1}(s+\theta_{k})\right\Vert
_{\mathcal{B}(\mathcal{X})}\leq\beta_{2}\left(  1+\alpha_{2}\right)
^{s-t},\text{ }s\geq t.
\]
Taking limit as $k\rightarrow\infty$ we get%
\begin{align*}
\left\Vert \bar{X}(t)\overline{P}\bar{X}^{-1}(s)\right\Vert _{\mathcal{B}%
(\mathcal{X})}  &  \leq\beta_{1}\left(  1+\alpha_{1}\right)  ^{s-t},\text{
}t\geq s,\\
\left\Vert \bar{X}(t)\left(  I-\overline{P}\right)  \bar{X}^{-1}(s)\right\Vert
_{\mathcal{B}(\mathcal{X})}  &  \leq\beta_{2}\left(  1+\alpha_{2}\right)
^{s-t},\text{ }s\geq t.
\end{align*}
This shows that the limiting system (\ref{lim sys}) admits exponential
dichotomy with the projection $\overline{P}$ and the positive constants
$\alpha_{1},\alpha_{2}$, $\beta_{1}$, and $\beta_{2}$. To prove (\ref{back1})
and (\ref{back2}), we can follow the similar procedure that we used to get
(\ref{forward limit1}) and (\ref{forward limit 2}). This completes the proof.
\end{proof}

\section{Existence results}

In this section, we propose some sufficient conditions for existence of almost
automorphic solutions of the nonlinear neutral delay difference system%
\begin{equation}
x(t+1)=A(t)x(t)+\Delta Q(t,x(t-g(t)))+G(t,x(t),x(t-g(t))), \label{neutral}%
\end{equation}
where $A(t)$ is an $n\times n$ matrix function, $g:\mathbb{Z\rightarrow
Z}_{+}$ is scalar, and the
functions $Q:\mathbb{Z\times}\mathcal{X}\mathbb{\rightarrow}\mathcal{X}$ and
$G:\mathbb{Z\times\mathcal{X}\times\mathcal{X}\rightarrow}\mathcal{X}$ are
continuous in $x$.

In our analysis, we use the following fixed point theorem:

\begin{theorem}
[Krasnoselskii]Let $\mathbb{M}$ be a closed, convex and nonempty subset of a
Banach space $\left(  \mathbb{B},\left\Vert .\right\Vert \right)  $. Suppose
that $H_{1}$ and $H_{2}$ map $\mathbb{M}$ into $\mathbb{B}$ such that

\begin{enumerate}
\item[i.] $x,y\in\mathbb{M}$ implies $H_{1}x+H_{2}y\in\mathbb{M}$,

\item[ii.] $H_{2}$ is continuous and $H_{2}\mathbb{M}$ contained in a compact set,

\item[iii.] $H_{1}$ is a contraction mapping.
\end{enumerate}

Then there exists $z\in\mathbb{M}$ with $z=H_{1}z+H_{2}z.$
\end{theorem}

Hereafter, we suppose that the following conditions hold:

\begin{description}
\item[A1] Functions $A(t),$ $g(t),$ $Q(t,u)$ and $G(t,u,v)$ are almost
automorphic in $t,$

\item[A2] For $\zeta,\psi\in\mathcal{A}(\mathcal{X})$, there exists a constant
$E_{1}>0$ such that%
\[
\left\Vert Q\left(  t,\zeta\right)  -Q\left(  t,\psi\right)  \right\Vert
_{\mathcal{X}}\leq E_{1}\left\Vert \zeta-\psi\right\Vert _{\mathcal{A}%
(\mathcal{X})}\text{ for all }t\in\mathbb{Z}\text{, }%
\]

\item[A3] For $\zeta,\psi\in\mathcal{A}(\mathcal{X})$, there exists a constant
$E_{2}>0$ such that%
\[
\left\Vert G\left(  t,u,\zeta\right)  -G\left(  t,u,\psi\right)  \right\Vert
_{\mathcal{X}}\leq E_{2}\left(  \left\Vert u-v\right\Vert _{\mathcal{X}%
}+\left\Vert \zeta-\psi\right\Vert _{\mathcal{A}(\mathcal{X})}\right)  \text{
for all }t\in\mathbb{Z}%
\]
and for any vector valued functions $u$ and $v$ defined on $\mathcal{X}.$

\item[A4] The homogeneous system (\ref{1}) admits an exponential dichotomy.
\end{description}

The following result can be proven similar to \cite[Lemma 2.4]{Chen}, hence we
omit it.

\begin{lemma}
If $u,$ $v:\mathbb{Z}\rightarrow\mathcal{X}$\ are almost automorphic
functions, then $u(t-v(t))$ is also discrete almost automorphic.
\end{lemma}

We say that $x:\mathbb{Z}\rightarrow\mathcal{X}$ is a solution of%
\[
x(t+1)=A(t)x(t)+f(t,x\left(  t\right)  )
\]
if it satisfies%
\[
x(t)=%
{\displaystyle\sum\limits_{j=-\infty}^{t-1}}
X(t)PX^{-1}(j+1)f(j,x\left(  j\right)  )-%
{\displaystyle\sum\limits_{j=t}^{\infty}}
X(t)\left(  I-P\right)  X^{-1}(j+1)f(j,x\left(  j\right)  ),
\]
where $X(t)$ is principal fundamental matrix solution of system (\ref{1}) (see
\cite[Definition 4.1]{Lizama2}).

Now, define the mapping $H$ by
\[
(Hx)(t):=(H_{1}x)(t)+(H_{2}x)(t),
\]
where%
\begin{equation}
(H_{1}x)(t):=Q(t,x(t-g(t))), \label{H1}%
\end{equation}
and%
\[
(H_{2}x)(t):=%
{\displaystyle\sum\limits_{j=-\infty}^{t-1}}
X(t)PX^{-1}(j+1)\Lambda(j,x)-%
{\displaystyle\sum\limits_{j=t}^{\infty}}
X(t)\left(  I-P\right)  X^{-1}(j+1)\Lambda(j,x),
\]
where $\Lambda(j,x)$ is given by%
\begin{equation}
\Lambda(j,x):=\left(  A(j)-I\right)  Q(j,x(j-g(j)))+G(j,x(j),x(j-g(j))).
\label{f}%
\end{equation}

\begin{lemma}
The mapping $H$ maps $\mathcal{A}(\mathcal{X})$ into $\mathcal{A}%
(\mathcal{X})$.
\end{lemma}

\begin{proof}
Suppose that $x\in\mathcal{A}(\mathcal{X})$. First, we deduce by using (A1-A3)
along with Theorem (\ref{thmlipschitz}) that the functions $Q$ and $G$ are
discrete almost automorphic. That is,%
\[
\lim_{n\rightarrow\infty}Q(t+k_{n},x(t+k_{n}-g(t+k_{n}))):=\overline
{Q}(t,\overline{x}(t-\overline{g}(t)))
\]
and%
\[
\lim_{n\rightarrow\infty}G(t+k_{n},x\left(  t+k_{n}\right)  ,x(t+k_{n}%
-g(t+k_{n}))):=\overline{G}(t,\overline{x}(t),\overline{x}(t-\overline{g}(t)))
\]
are well defined for each $t\in\mathbb{Z}$ and%
\[
\lim_{n\rightarrow\infty}\overline{Q}(t-k_{n},\overline{x}(t-k_{n}-\overline{g}(t-k_{n}%
)))=Q(t,x(t-g(t)))
\]
and%
\[
\lim_{n\rightarrow\infty}\overline{G}(t-k_{n},\overline{x}\left(  t-k_{n}\right)
,\overline{x}(t-k_{n}-\overline{g}(t-k_{n})))=G(t,x(t),x(t-g(t)))
\]
for each $t\in\mathbb{Z}$. Second, we have%
\begin{align*}
(Hx)(t+k_{n})  &  =Q(t+k_{n},x(t+k_{n}-g(t+k_{n})))+%
{\displaystyle\sum\limits_{j=-\infty}^{t-1}}
X(t+k_{n})PX^{-1}(j+k_{n}+1)\Lambda(j+k_{n},x)\\
&  -%
{\displaystyle\sum\limits_{j=t}^{\infty}}
X(t+k_{n})\left(  I-P\right)  X^{-1}(j+k_{n}+1)\Lambda(j+k_{n},x).
\end{align*}
Taking the limit $n\rightarrow\infty$ and employing Lebesgue convergence
theorem, we conclude that%
\begin{align*}
\overline{(Hx)}(t)  &  :=\lim_{n\rightarrow\infty}(Hx)(t+k_{n})=\overline
{Q}(t,\overline{x}(t-\overline{g}(t)))+%
{\displaystyle\sum\limits_{j=-\infty}^{t-1}}
\overline{X}(t)\overline{P}\overline{X}^{-1}(j+1)\overline{\Lambda}(j,x)\\
&  -%
{\displaystyle\sum\limits_{j=t}^{\infty}}
\overline{X}(t)\left(  I-\overline{P}\right)  \overline{X}^{-1}(j+1)\overline{\Lambda}(j,x),
\end{align*}
is well defined for each $t\in\mathbb{Z}$, where%
\[
\overline{\Lambda}(j,x):=\left(  \overline{A}(j)-I\right)  \overline{Q}(j,\overline{x}%
(j-\overline{g}(j)))+\overline{G}(j,\overline{x}(j),\overline{x}(j-\overline{g}(j))).
\]
Applying similar procedure to the following%
\begin{align*}
\overline{(Hx)}(t-k_{n})  &  =\overline{Q}(t-k_{n},\overline{x}(t-k_{n}-\overline{g}%
(t-k_{n})))+%
{\displaystyle\sum\limits_{j=-\infty}^{t-1}}
\overline{X}(t-k_{n})\overline{P}\overline{X}^{-1}(j-k_{n}+1)\overline{\Lambda}(j-k_{n})\\
&  -%
{\displaystyle\sum\limits_{j=t}^{\infty}}
\overline{X}(t-k_{n})\left(  I-\overline{P}\right)  \overline{X}^{-1}(j-k_{n}+1)\overline
{\Lambda}(j-k_{n},x).
\end{align*}
we get%
\[
\lim_{n\rightarrow\infty}\overline{(Hx)}(t-k_{n})=(Hx)(t),
\]
for each $t\in\mathbb{Z}$. This means $Hx\in\mathcal{A}(\mathcal{X})$. This
completes the proof.
\end{proof}

The following result is a direct consequence of (A2) and (\ref{H1}).

\begin{lemma}
Assume (A2). If $E_{1}<1,$ then the operator $H_{1}$ is a contraction.
\end{lemma}

\begin{lemma}
Assume (A1-A4). Define the set
\[
\Pi_{M}:=\left\{  x\in\mathcal{A}(\mathcal{X}),\left\Vert x\right\Vert
_{\mathcal{A}(\mathcal{X})}\leq M\right\}
\]
where $M$ is a fixed constant. The operator $H_{2}$ is continuous and the
image$\ H_{2}\left(  \Pi_{M}\right)  $ is contained in a compact set.
\end{lemma}

\begin{proof}
By (A4), we have the following:%
\begin{equation}
\left\Vert (H_{2}x)\right\Vert _{\mathcal{A}(\mathcal{X})}\leq\left\Vert
f(.,x(.)\right\Vert _{\mathcal{A}(\mathcal{X})}\left[  \beta_{1}\frac
{1+\alpha_{1}}{\alpha_{1}}+\frac{\beta_{2}}{\alpha_{2}}\right]
,\label{inequality}%
\end{equation}
where $f$ is defined by (\ref{f}). To see that $H_{2}$ is continuous, suppose
$\zeta$, $\psi\in\mathcal{A}(\mathcal{X})$ and define the number
$\delta(\varepsilon)>0$ by%
\[
\delta:=\frac{\varepsilon}{\left[  \left(  \left\Vert A\right\Vert +1\right)
E_{1}\left\Vert \zeta-\psi\right\Vert _{\mathcal{A}(\mathcal{X})}%
+2E_{2}\left\Vert \zeta-\psi\right\Vert _{\mathcal{A}(\mathcal{X})}\right]
\left(  \beta_{1}\frac{1+\alpha_{1}}{\alpha_{1}}+\frac{\beta_{2}}{\alpha_{2}%
}\right)  },
\]
for any given $\varepsilon>0$ where
\[
\left\Vert A\right\Vert =\sup_{t\in\mathbb{Z}}\left\vert A\left(  t\right)
\right\vert
\]
and%
\[
\left\vert A(t)\right\vert :=\max_{1\leq i\leq n}%
{\displaystyle\sum\limits_{j=1}^{n}}
\left\vert a_{ij}\left(  t\right)  \right\vert .
\]
If $\left\Vert \zeta-\psi\right\Vert _{\mathcal{A}(\mathcal{X})}<\delta$, then
we have%
\begin{align*}
\left\Vert H_{2}\left(  \zeta\right)  (t)-H_{2}\left(  \psi\right)
(t)\right\Vert _{\mathcal{X}} &  \leq%
{\displaystyle\sum\limits_{j=-\infty}^{t-1}}
\left(  \left\Vert X(t)PX^{-1}\left(  j+1\right)  \right\Vert _{\mathcal{B}%
(\mathcal{X})}\right.  \\
&  \left.  \times\left[  \left(  \left\Vert A\right\Vert +1\right)  \left\Vert
Q\left(  j,\zeta\left(  j-g(j)\right)  \right)  -Q\left(  s,\psi\left(
j-g(j)\right)  \right)  \right\Vert _{\mathcal{X}}\right.  \right.  \\
&  \left.  +\left.  \left\Vert G\left(  j,\zeta\left(  j\right)  ,\zeta\left(
j-g(j)\right)  \right)  -G\left(  j,\psi\left(  j\right)  ,\psi\left(
j-g(j)\right)  \right)  \right\Vert _{\mathcal{X}}\right]  \right)  \\
&  +%
{\displaystyle\sum\limits_{j=t}^{\infty}}
\left(  \left\Vert X(t)\left(  I-P\right)  X^{-1}\left(  j+1\right)
\right\Vert _{\mathcal{B}%
(\mathcal{X})}\right.  \\
&  \left.  \times\left[  \left(  \left\Vert A\right\Vert +1\right)  \left\Vert
Q\left(  j,\zeta\left(  j-g(j)\right)  \right)  -Q\left(  s,\psi\left(
j-g(j)\right)  \right)  \right\Vert _{\mathcal{X}}\right.  \right.  \\
&  \left.  +\left.  \left\Vert G\left(  j,\zeta\left(  j\right)  ,\zeta\left(
j-g(j)\right)  \right)  -G\left(  j,\psi\left(  j\right)  ,\psi\left(
j-g(j)\right)  \right)  \right\Vert _{\mathcal{X}}\right]  \right)  .
\end{align*}
By (vii) of Theorem \ref{thm properties of almost automorphic} and (A2-A4), we
get%
\begin{align*}
\left\Vert H_{2}\left(  \zeta\right)  (t)-H_{2}\left(  \psi\right)
(t)\right\Vert _{\mathcal{X}} &  \leq%
{\displaystyle\sum\limits_{j=-\infty}^{t-1}}
\beta_{1}\left(  1+\alpha_{1}\right)  ^{j+1-t}\left[  \left(  \left\Vert
A\right\Vert +1\right)  E_{1}\left\Vert \zeta-\psi\right\Vert _{\mathcal{A}%
(\mathcal{X})}+2E_{2}\left\Vert \zeta-\psi\right\Vert _{\mathcal{A}%
(\mathcal{X})}\right]  \\
&  +%
{\displaystyle\sum\limits_{j=t}^{\infty}}
\beta_{2}\left(  1+\alpha_{2}\right)  ^{t-j-1}\left[  \left(  \left\Vert
A\right\Vert +1\right)  E_{1}\left\Vert \zeta-\psi\right\Vert _{\mathcal{A}%
(\mathcal{X})}+2E_{2}\left\Vert \zeta-\psi\right\Vert _{\mathcal{A}%
(\mathcal{X})}\right]  \\
&  \leq\left[  \left(  \left\Vert A\right\Vert +1\right)  E_{1}\left\Vert
\zeta-\psi\right\Vert _{\mathcal{A}(\mathcal{X})}+2E_{2}\left\Vert \zeta
-\psi\right\Vert _{\mathcal{A}(\mathcal{X})}\right]  \left(  \beta_{1}%
\frac{1+\alpha_{1}}{\alpha_{1}}+\frac{\beta_{2}}{\alpha_{2}}\right)  \\
&  <\varepsilon,
\end{align*}
which shows that $H_{2}$ is continuous. \newline Now, we show that $H_{2}%
(\Pi_{M})$ is contained in a compact set. For any $\zeta,\psi\in\Pi_{M}$ we
have%
\begin{align*}
\left\Vert G(t,\zeta(t),\psi\left(  t-g(t)\right)  )\right\Vert _{\mathcal{X}}
&  \leq\left\Vert G(t,\zeta(t),\psi\left(  t-g(t)\right)
)-G(t,0,0)\right\Vert _{\mathcal{X}}+\left\Vert G(t,0,0)\right\Vert
_{\mathcal{X}}\\
&  \leq E_{2}\left(  \left\Vert \zeta\right\Vert _{\mathcal{A}(\mathcal{X}%
)}+\left\Vert \psi\right\Vert _{\mathcal{A}(\mathcal{X})}\right)  +a\\
&  \leq2ME_{2}+a,
\end{align*}
and%
\begin{align*}
\left\Vert Q\left(  t,\zeta\left(  t-g(t)\right)  \right)  \right\Vert
_{\mathcal{X}} &  \leq\left\Vert Q\left(  t,\zeta\left(  t-g(t)\right)
\right)  -Q(t,0)\right\Vert _{\mathcal{X}}+\left\Vert Q(t,0)\right\Vert
_{\mathcal{X}}\\
&  \leq E_{1}\left\Vert \zeta\right\Vert _{\mathcal{A}(\mathcal{X})}+b\\
&  \leq E_{1}M+b
\end{align*}
where $a:=\left\Vert G(t,0,0)\right\Vert _{\mathcal{X}}$ and $b:=\left\Vert
Q(t,0)\right\Vert _{\mathcal{X}}$. This implies%
\[
\left\Vert H_{2}\left(  \zeta_{n}\right)  (t)\right\Vert _{\mathcal{A}%
(\mathcal{X})}\leq\left[  \left(  \left\Vert A\right\Vert +1\right)  \left(
E_{1}M+b\right)  +2E_{2}M+a\right]  \left[  \beta_{1}\frac{1+\alpha_{1}%
}{\alpha_{1}}+\frac{\beta_{2}}{\alpha_{2}}\right]
\]
for any sequence $\{\zeta_{n}\}$ in $\Pi_{M}$. Moreover, from (A1),(A4) and
(\ref{inequality}), we deduce that $\Delta\left(  H_{2}(\zeta_{n}(t))\right)
$ is bounded. That means, $H_{2}(\zeta_{n})$ is uniformly bounded an
equicontinuous. The proof follows from Arzela-Ascoli theorem.
\end{proof}

\begin{theorem}
\label{final theorem}Assume (A1-A4). Let $M_{0}$ be a constant satisfying the
following inequality%
\[
E_{1}M_{0}+b+\left[  \left(  \left\Vert A\right\Vert +1\right)  \left(
E_{1}M_{0}+b\right)  +2E_{2}M_{0}+a\right]  \left[  \beta_{1}\frac
{1+\alpha_{1}}{\alpha_{1}}+\frac{\beta_{2}}{\alpha_{2}}\right]  \leq M_{0},
\]
where $E_{1}\in\left(  0,1\right)  $ and%
\[
a:=\left\Vert G(t,0,0)\right\Vert _{\mathcal{X}},b:=\left\Vert
Q(t,0)\right\Vert _{\mathcal{X}}.
\]
Then the equation (\ref{neutral}) has an almost automorphic solution in
$\Pi_{M_{0}}$.
\end{theorem}
\begin{proof}
For $\psi\in\Pi_{M_{0}}$, we have%
\begin{align*}
\left\Vert H_{1}(\psi(t))+H_{2}(\psi(t))\right\Vert _{\mathcal{X}} &
\leq\left\Vert Q\left(  t,\psi\left(  t-g(t)\right)  \right)
-Q(t,0)\right\Vert _{\mathcal{X}}+\left\Vert Q(t,0)\right\Vert _{\mathcal{X}%
}\\
&  +%
{\displaystyle\sum\limits_{j=-\infty}^{t-1}}
\left\{  \left\Vert X(t)PX^{-1}\left(  j+1\right)  \right\Vert _{\mathcal{B}%
(\mathcal{X})}\right.  \\
&  \times\left.  \left\Vert \left(  A(j)-I\right)  Q\left(  j,x\left(
j-g(j)\right)  \right)  +G\left(  j,\psi\left(  j\right)  ,\psi\left(
j-g(j)\right)  \right)  \right\Vert _{\mathcal{X}}\right\}  \\
&  +%
{\displaystyle\sum\limits_{j=t}^{\infty}}
\left\{  \left\Vert X(t)\left(  I-P\right)  X^{-1}\left(  j+1\right)
\right\Vert _{\mathcal{B}(\mathcal{X})}\right.  \\
&  \times\left.  \left\Vert \left(  A(j)-I\right)  Q\left(  j,x\left(
j-g(j)\right)  \right)  +G\left(  j,\psi\left(  j\right)  ,\psi\left(
j-g(j)\right)  \right)  \right\Vert _{\mathcal{X}}\right\}  \\
&  \leq E_{1}M_{0}+b+\left[  \left(  \left\Vert A\right\Vert +1\right)
\left(  E_{1}M_{0}+b\right)  +2E_{2}M_{0}+a\right]  \left[  \beta_{1}%
\frac{1+\alpha_{1}}{\alpha_{1}}+\frac{\beta_{2}}{\alpha_{2}}\right]  \\
&  \leq M_{0}%
\end{align*}
which means $H_{1}(\psi)+H_{2}(\psi)\in\Pi_{M_{0}}$. Then all conditions of
fixed point theorem are satisfied and there exists a $x\in\Pi_{M_{0}}$ such
that $x(t)=H_{1}(x(t))+H_{2}(x(t)).$ The proof is complete.
\end{proof}

\begin{example}
Let the neutral delay discrete system be given by%
\begin{equation}
x(t+1)=\frac{1}{3}sgn\left(  \cos2\pi t\theta\right)  x(t)I+\frac{1}{10}\Delta
x\left(  t-\tau\right)  +%
\begin{bmatrix}
\sin(\frac{\pi}{2}t)+\sin(\frac{\pi}{2}t\sqrt{2})\\
\cos\pi t+\cos\pi t\sqrt{2}%
\end{bmatrix}
+\frac{1}{20}x\left(  t-\tau\right)  ,\label{final ex}%
\end{equation}
where $\theta$ is an irrational number, $\tau$ is a positive integer with
$t>\tau$ and Banach space $\mathcal{X}=\mathbb{R}.$ In \cite{Veeech}, it is
shown that $sgn\left(  \cos2\pi t\theta\right)  $ is an almost automorphic
function for $t\in\mathbb{Z}$ and $\theta$ is irrational. Therefore, the
matrix function%
\[
A(t)=%
\begin{bmatrix}
\frac{1}{3}sgn\left(  \cos2\pi t\theta\right)   & 0\\
0 & \frac{1}{3}sgn\left(  \cos2\pi t\theta\right)
\end{bmatrix}
\]
is discrete almost automorphic. Comparing (\ref{final ex}) with (\ref{neutral}%
), we have vector functions
\[
Q(t,x(t-g(t))))=%
\begin{bmatrix}
\frac{1}{10}x_{1}\left(  t-\tau\right)  \\
\frac{1}{10}x_{2}\left(  t-\tau\right)
\end{bmatrix}
\]
and%
\[
G(t,x(t),x(t-g(t)))=%
\begin{bmatrix}
\sin(\frac{\pi}{2}t)+\sin(\frac{\pi}{2}t\sqrt{2})+\frac{1}{20}x_{1}\left(  t-\tau\right)  \\
\cos\pi t+\cos\pi t\sqrt{2}+\frac{1}{20}x_{2}\left(  t-\tau\right)
\end{bmatrix}
,
\]
which are discrete almost automorphic. Then assumption (A1) is satisfied. For
any $\varsigma$, $\psi\in\Pi_{M_{0}}$, we have%
\[
\left\vert Q(t,\varsigma(t-g(t)))-Q(t,\psi(t-g(t)))\right\vert \leq\frac
{1}{10}\left\Vert \varsigma-\psi\right\Vert _{A(\mathbb{R})}%
\]
and%
\[
\left\vert G(t,\varsigma(t),\varsigma(t-g(t)))-G(t,\psi(t),\psi
(t-g(t)))\right\vert \leq\frac{1}{20}\left\Vert \varsigma-\psi\right\Vert
_{A(\mathbb{R})}.
\]
Then (A2-A3) hold with $E_{1}=\frac{1}{10},$ $E_{2}=\frac{1}{20}$, $a=2$ and
$b=0$.\newline By using Putzer algorithm (see \cite[Theorem 5.35]{bohner})
with $P$-matrices under the special case $\mathbb{T=Z}$, we get $P_{0}%
=I_{2\times2}$ and $P_{1}=0_{2\times2}.$ Then the principal fundamental matrix
solution of the homogeneous system%
\[
x(t+1)=\frac{1}{3}sgn\left(  \cos2\pi t\theta\right)  x(t)I
\]
can be written as
\[
X\left(  t\right)  =%
\begin{bmatrix}
3^{-t}\left(
{\displaystyle\prod\limits_{j=0}^{t-1}}
sgn\left(  \cos2\pi j\theta\right)  \right)   & 0\\
0 & 3^{-t}\left(
{\displaystyle\prod\limits_{j=0}^{t-1}}
sgn\left(  \cos2\pi j\theta\right)  \right)
\end{bmatrix}
.
\]
\newline Since%
\[
\left\vert 3^{s-t}\left(
{\displaystyle\prod\limits_{j=s}^{t-1}}
sgn\left(  \cos2\pi j\theta\right)  \right)  \right\vert =3^{s-t}\leq\beta
_{1}\left(  1+\alpha_{1}\right)  ^{s-t}\text{ for }t\geq s
\]
is satisfied for $\beta_{1}=1$ and $\alpha_{1}=1$, the homogeneous system
admits exponential dichotomy, as desired. Moreover, we may assume $\alpha
_{1}=\alpha_{2}$ and $\beta_{1}=$ $\beta_{2}$ since $P_{1}=0_{2\times2}.$ That
is, all assumptions of Theorem \ref{final theorem} hold. Hence, we conclude
that the system (\ref{final ex}) has an almost automorphic solution in
$\Pi_{M_{0}}$ whenever $M_{0}$ satisfies the inequality%
\[
\frac{1}{10}M_{0}+\frac{4}{10}M_{0}+\frac{3}{10}M_{0}+6\leq M_{0}%
\]
or equivalently%
\[
30\leq M_{0}.
\]

\end{example}

The following existence result is given in \cite{Lizama2}:

\begin{theorem}
\cite[Theorem 4.3]{Lizama2} Suppose $A\left(  k\right)  $ is discrete almost
automorphic and a non-singular matrix and the set $\left\{  A^{-1}\left(
k\right)  \right\}  _{k\in\mathbb{Z}}$ is bounded. Also, assume the
homogeneous system $U(k+1)=A\left(  k\right)  U\left(  k\right)  $,
$k\in\mathbb{Z}$, admits an exponential dichotomy on $\mathbb{Z}$ with
positive constants $\eta$, $\nu$, $\beta$, $\alpha$ and the function
$f:\mathbb{Z}\times E^{n}\rightarrow E^{n}$ is discrete almost automorphic in
$k$ for each $u$ in $E^{n}$, satisfying the following condition:

\begin{enumerate}
\item There exists a constant $0<L<\frac{(1-e^{-\alpha})(e^{\beta}-1)}%
{\eta(e^{\beta}-1)+\nu(1-e^{-\alpha})}$ such that%
\[
\left\Vert f\left(  k,u\right)  -f\left(  k,v\right)  \right\Vert \leq
L\left\Vert u-v\right\Vert
\]
for every $u,v\in E^{n}$ and $k\in\mathbb{Z}$. Then the system%
\[
U(k+1)=A\left(  k\right)  U\left(  k\right)  +f\left(  k,u\left(  k\right)
\right)  ,k\in\mathbb{Z}%
\]
has a unique almost automorphic solution.
\end{enumerate}
\end{theorem}

\begin{example}
\label{ex comparison} The conditions of our existence result are weaker than
the conditions of \cite[Theorem 4.3]{Lizama2}. In \cite[Theorem 4.3]{Lizama2},
the authors require boundedness of the inverse matrix $A^{-1}\left(  t\right)
$ to deduce existence of almost automorphic solutions of the system
\[
x\left(  t+1\right)  =A\left(  t\right)  x\left(  t\right)  +f\left(
t,x\right)  .
\]
In particular, \cite[Theorem 4.3]{Lizama2} is invalid for the system%
\begin{equation}
x(t+1)=%
\begin{bmatrix}
\frac{1}{2}\sin\left(  \frac{\pi}{2}t\right)   & 0\\
0 & \frac{1}{2}\sin\left(  \frac{\pi}{2}t\right)
\end{bmatrix}
x\left(  t\right)  +f(t,x),\ \ t\in\mathbb{Z}\label{ex2}%
\end{equation}
since the matrix
\[
A\left(  t\right)  =%
\begin{bmatrix}
\frac{1}{2}\sin\left(  \frac{\pi}{2}t\right)   & 0\\
0 & \frac{1}{2}\sin\left(  \frac{\pi}{2}t\right)
\end{bmatrix}
\]
is singular for some integers. However, Theorem \ref{final theorem} implies the existence of
discrete almost automorphic solution of the system (\ref{ex2}) for an almost
automorphic function $f(t,x)$ satisfying (A1) and (A3).
\end{example}

One may repeat the same procedure in the last section by replacing
$\mathcal{A}\left(  \mathcal{X}\right)  $ with $\mathcal{AP}\left(
\mathcal{X}\right)  $, the space of all almost periodic functions on
$\mathcal{X}$, and the assumption (A1) with the following\newline

\begin{description}
\item[A1$^{\prime}$] Functions $A(t),$ $g(t),$ $Q(t,u)$ and $G(t,u,v)$ are
almost periodic in $t$\newline
\end{description}

\noindent to arrive at the following result:

\begin{theorem}
[Almost periodic solutions of the system (\ref{neutral})]Assume (A1$^{\prime}%
$) and (A2-A4). Let $M_{0}$ be a constant satisfying the following inequality%
\[
E_{1}M_{0}+b+\left[  \left(  \left\Vert A\right\Vert +1\right)  \left(
E_{1}M_{0}+b\right)  +2E_{2}M_{0}+a\right]  \left[  \beta_{1}\frac
{1+\alpha_{1}}{\alpha_{1}}+\frac{\beta_{2}}{\alpha_{2}}\right]  \leq M_{0},
\]
where $E_{1}\in\left(  0,1\right)  $ and%
\[
a:=\left\Vert G(t,0,0)\right\Vert _{\mathcal{X}},b:=\left\Vert
Q(t,0)\right\Vert _{\mathcal{X}}.
\]
Then the equation (\ref{neutral}) has an almost periodic solution in
$\widetilde{\Pi}_{M_{0}}:=\left\{  x\in\mathcal{AP}(\mathcal{X}),\left\Vert
x\right\Vert _{\mathcal{AP}(\mathcal{X})}\leq M_{0}\right\}  $.
\end{theorem}

\end{document}